\documentclass[11pt]{article}
\usepackage{caption}
\usepackage{subfigure}
\usepackage{indentfirst}
\usepackage{latexsym}
\usepackage{hyperref}
\usepackage{graphicx}
\usepackage{mathrsfs}
\usepackage{bookmark}
\usepackage{enumerate}
\usepackage{amsfonts,amsthm,amssymb,mathtools}
\usepackage{amsmath}
\usepackage{latexsym, euscript, epic, eepic, color}
\usepackage{amssymb}

\usepackage{enumerate}
\usepackage[all]{xy}
\usepackage{setspace}
\allowdisplaybreaks

\usepackage{epstopdf}
\numberwithin{equation}{section}
\newtheorem{theorem}{Theorem}[section]
\newtheorem{corollary}[theorem]{Corollary}
\newtheorem{definition}[theorem]{Definition}

\newtheorem{claim}[theorem]{Claim}
\newtheorem{lemma}[theorem]{Lemma}

 \allowdisplaybreaks

\begin{document}
\baselineskip=16pt

\title{On the $\alpha$-spectral radius of the $k$-uniform supertrees}

\author{\textbf{Chang Liu, Jianping Li\textsuperscript{\thanks{Corresponding author.}}}\\
	\small \emph{College of Liberal Arts and Sciences, National University of Defense Technology,}\\[-0.8ex]
	\small \emph{Changsha, Hunan, 410073, P. R. China}\\[-0.4ex]
		\small\tt clnudt19@126.com\\
		\small\tt lijianping65@nudt.edu.cn\\
	}

\date{\today}

\maketitle

\begin{abstract}
	Let $G$ be a $k$-uniform hypergraph with vertex set $V(G)$ and edge set $E(G)$. A connected and acyclic hypergraph is called a supertree. For $0\leq\alpha<1$, the $\alpha$-spectral radius of $G$ is the largest $H$-eigenvalue of $\alpha D(G)+(1-\alpha)A(G)$, where $D(G)$ and $A(G)$ are the diagonal tensor of the degrees and the adjacency tensor of $G$, respectively. In this paper, we determine the unique supertrees with the maximum $\alpha$-spectral radius among all $k$-uniform supertrees with $m$ edges and independence number $\beta$ for $\lceil\frac{m(k-1)+1}{k}\rceil\leq\beta\leq m$, among all $k$-uniform supertrees with given degree sequences, and among all $k$-uniform supertrees with $m$ edges and matching number $\mu$ for $1\leq\mu\leq\lfloor\frac{m(k-1)+1}{k}\rfloor$, respectively.
	\\[2pt]
	\textbf{AMS Subject Classification:} 05C50; 05C65; 15A42\\[2pt]
	\textbf{Keywords:} $\alpha$-spectral radius; uniform hypergraph; supertree; independence number; degree sequence; matching number.
\end{abstract}

\section{Introduction}
Let $G$ be a hypergraph on $n$ vertices with vertex set $V(G)=\{v_1,v_2,\dots,v_n\}$ and edge set $E(G)$. The element of $E(G)$ is a subset of $V(G)$, containing at least two vertices. If $|e|=k$ for each $e\in E(G)$, $G$ refers the $k$-uniform hypergraph. For $k=2$, $G$ is the ordinary graph. The degree of a vertex $v\in G$, denoted by $d_{G}(v)$, is defined as the number of edges in $E(G)$ containing $v$. For an edge $e=\{v_1,\dots,v_k\}$ in a $k$-uniform hypergraph $G$, if $d_{G}(v_1)\geq2$, $d_{G}(v_i)=1$ for $2\leq i\leq k$, we say that $e$ is a pendent edge of $G$ (at $v_1$). A vertex of degree one is a pendent vertex.

In a hypergraph, we say that two vertices $u,v\in V(G)$ are adjacent if there exists an edge $e\in E(G)$ such that $\{u,v\}\in e$, and say that two edges $e,f\in E(G)$ are adjacent if $e\cap f\ne\emptyset$. A path of length $l$ from $u$ to $v$ in $G$ is an alternating sequence of vertices and edges $u=v_1,e_1,\dots,v_{l},e_l,v_{l+1}=v$ such that $\{v_i,v_{i+1}\}\subseteq e_i$ for $1\leq i\leq l$, where  $v_1,v_2,\dots,v_{l+1}$ are distinct vertices and $e_1,e_2,\dots,e_l$ are distinct edges. For any $u,v\in V(G)$, if there exists a path from $u$ to $v$, then we say that $G$ is connected. The distance between two vertices $u$ and $v$ is the minimum length of a path in $G$ which connects $u$ and $v$, denoted by $d_{G}(u,v)$. A cycle of length $l$ is an alternating sequence $v_1,e_1,v_2,e_2,\dots,v_l,e_l,v_1$ of distinct vertices and edges such that $\{v_i,v_{i+1}\}\subseteq e_i$ for $1\leq i\leq l-1$ and $\{v_l,v_1\}\subseteq e_l$. If a hypergraph $G$ is both connected and acyclic, then we say that $G$ is a supertree. Two edges in a supertree share at most one common vertex. 

An independent set $S$ of a hypergraph $G$ is a subset of $V(G)$ such that any two vertices in $S$ are not contained in an edge of $G$. If for any other independent set $S'$ of $G$ such that $|S'|\leq|S|$, we say that $S$ is a maximum independent set of $G$. The independent number is the size of a maximum independent of $G$, denoted by $\beta(G)$ (or $\beta$ for short), i.e., $\beta(G)=\max\{|S|: S\mbox{ is a independent set of }G\}$. A matching $M$ of hypergraph $G$ is a subset of $E(G)$ which satisfies that any two edges in $M$ contain no same vertex. The matching number is the maximum of the cardinalities of all matchings of $G$, written as $\mu(G)$ (or $\mu$ for short), i.e., $\mu(G)=\max\{|M|:M\mbox{ is a matching of }G\}$. 

For positive integers $k$ and $n$, a tensor $\mathcal{T}=(\mathcal{T}_{i_1\dots i_k})$ of order $k$ and dimension $n$ refers to a multidimensional array with complex entries $\mathcal{T}_{i_1\dots i_k}$ for $i_j\in[n]=\{1,\dots,n\}$ and $j\in[k]$. A tensor $\mathcal{T}$ is called symmetric if its entry $\mathcal{T}_{i_1\dots i_k}$ is invariant under any permutation of its indices $i_1,\dots,i_k$.

In 2005, Qi \cite{Qi1} and Lim \cite{Lim1} independently introduced the definition of tensor eigenvalues. For a tensor $\mathcal{T}=(\mathcal{T}_{i_1\dots i_k})$ of order $k$ and dimension $n$ and a vector $x=(x_1,\dots,x_n)^\top$, $\mathcal{T}x^{k-1}$ is defined as a vector in $\mathbb{C}^n$ whose $i$-th entry equals
\begin{equation*}
	(\mathcal{T}x^{k-1})_i=\sum_{i_2,\dots,i_k=1}^{n}\mathcal{T}_{ii_2\dots i_k}x_{i_2}\cdots x_{i_k},
\end{equation*}
where $i\in[n]$. If there exists a number $\lambda\in\mathbb{C}$ and a nonzero vector $x\in\mathbb{C}^n$ such that 
\begin{equation*}
	\mathcal{T}x^{k-1}=\lambda x^{[k-1]},
\end{equation*} 
where $ x^{[k-1]}=\left( x_{1}^{k-1},\dots,x_{n}^{k-1}\right)^\top $, then $\lambda$ is called an eigenvalue of $\mathcal{T}$, and $x$ is called an eigenvector of $\mathcal{T}$ corresponding to the eigenvalue $\lambda$. By the definition of the general product of tensors \cite{Shao}, thus $\mathcal{T}x^{k-1}$ can simply written as $\mathcal{T}x$. Moreover, if both $\lambda$ and $x$ are real, then $\lambda$ is called an $H$-eigenvalue and $x$ is called an $H$-eigenvector of $\mathcal{T}$. The spectral radius of $\mathcal{T}$ is the maximum modulus of its eigenvalues, i.e., $\rho(\mathcal{T})=\max\{|\lambda|:\lambda \mbox{ is an eigenvalue of } \mathcal{T}\}$.

The adjacency tensor of a $k$-uniform hypergraph $G$ is defined as the order $k$ and dimension $n$ tensor $\mathcal{A}(G)=(a_{i_1\dots i_k})$, whose $(i_1\dots i_k)$-entry is
\begin{equation*}
	a_{i_1\dots i_k}=\begin{cases}
		\dfrac{1}{(k-1)!}, & \mbox{if }\{v_{i_1},\dots,v_{i_k}\}\in E(G), \\
		0, & \mbox{otherwise}.
	\end{cases}
\end{equation*}
See \cite{CoandDu,Lim2,PeandZh,XieandCh} for more details. In 2014, Qi \cite{Qi2} introduced a simple and natural definition for the Laplacian tensor $\mathcal{L}(G)$ and the signless Laplacian tensor $\mathcal{Q}(G)$ as $\mathcal{L}(G)=\mathcal{D}(G)-\mathcal{A}(G)$ and $\mathcal{Q}(G)=\mathcal{D}(G)+\mathcal{A}(G)$, where $\mathcal{A}(G)$ is the adjacency tensor of the hypergraph $G$ and $D(G)$ is the degree diagonal tensor of the hypergraph $G$. 

Motivated by the innovative work of Nikiforov \cite{Ni1}, Lin et al. \cite{Linandzh} proposed to study the convex linear combinations $\mathcal{A}_{\alpha}(G)$ of $\mathcal{D}(G)$ and $\mathcal{A}(G)$ defined by
\begin{equation*}
	\mathcal{A}_{\alpha}(G)=\alpha\mathcal{D}(G)+(1-\alpha)\mathcal{A}(G),
\end{equation*}
where $0\leq\alpha<1$. It's clear that $\mathcal{A}_0(G)$ is the adjacency tensor and $\mathcal{A}_{\frac{1}{2}}(G)$ is essentially equivalent to signless Laplacain tensor. The $\alpha$-spectral radius of $G$ is the spectral radius of $\mathcal{A}_{\alpha}(G)$, denoted by $\rho_{\alpha}(G)$. For a connected $k$-uniform hypergraph $G$, $\mathcal{A}_{\alpha}(G)$ is real and symmetric, thus $\rho_{\alpha}(G)$ is the largest $H$-eigenvalue of $\mathcal{A}_{\alpha}(G)$. In this case, $\mathcal{A}_{\alpha}(G)$  is weakly irreducible and there is a unique $k$-unit positive $H$-eigenvector $x$ corresponding to $\rho_{\alpha}(G)$ (see \cite{Guo and Zhou}), we call $x$ the $\alpha$-Perron vector of $G$. 

There have been many studies on the spectral radii of $k$-uniform hypergraph $G$ in the past decades. Li et al. \cite{LiShao and Qi} introduced the operation "moving edges" for hypergraphs and proved that the hyperstar attains uniquely the largest spectral radius and the largest signless Laplacian spectral radius among all $k$-uniform supertrees. Xiao et al. \cite{XiaoWan and Lu} determined the unique supertree with the maximum spectral radius among all $k$-uniform supertrees with given degree sequences. Xiao et al. \cite{XiaoWan and Du} characterize the first two largest spectral radii of $k$-uniform supertrees with given diameter. Guo and Zhou \cite{GuoanZhou srh} characterized the extremal supertrees with maximum spectral radius among all $k$-uniform supertrees with a given independent number or a given matching number. More works can be referred to \cite{Kang and Yuan,OuyangQi,Yuan Si}.

Recently, the research on extremal $\alpha$-spectral radius of uniform hypergraphs has attracted much attention. Lin et al. \cite{Linandzh} proposed some upper bounds on $\alpha$-spectral radius of connected irregular $k$-uniform hypergraphs. Guo and Zhou \cite{Guo and Zhou} introduced some transformations which increase the $\alpha$-spectral radius, and determined the unique hypergraph with the largest $\alpha$-spectral radius in some classes of uniform hypergraphs. Wang et al. \cite{DouWang} showed how the $\alpha$-spectral changes under the edge grafting operations on connected $k$-uniform hypergraphs, and obtained the extremal supertree for $\alpha$-spectral radius among $k$-uniform non-caterpillar hypergraphs with given order, size and diameter. In 2022, Kang et al. \cite{Kang and Shan} determined the unique unicyclic hypergraphs with the maximum $\alpha$-spectral radius among all $k$-uniform unicyclic hypergraphs with fixed diameter, and among all $k$-uniform unicyclic hypergraphs with given number of pendent edges. 

Motivated by \cite{GuoanZhou srh,Guo and Zhou,LiShao and Qi,XiaoWan and Lu}, we want to establish more extremal properties of the $\alpha$-spectral radius of $k$-uniform supertrees. In this paper, we determine the unique supertrees with the maximum $\alpha$-spectral radius among all $k$-uniform supertrees with $m$ edges and independence number $\beta$ for $\lceil\frac{m(k-1)+1}{k}\rceil\leq\beta\leq m$, among all $k$-uniform supertrees with given degree sequences, and among all $k$-uniform supertrees with $m$ edges and matching number $\mu$ for $1\leq\mu\leq\lfloor\frac{m(k-1)+1}{k}\rfloor$, respectively.

\section{Preliminaries}

Firstly, we introduce the formula for calculating the spectral radius of a nonnegative symmetric tensor $\mathcal{T}$ of order $k$ and dimension $n$, which is
\begin{equation*} 
	\rho(\mathcal{T})=\max\{x^{\top}(\mathcal{T}x)|x\in\mathbb{R}^{n}_{+},\Vert x\Vert_k=1\}.
\end{equation*}

For the $\alpha$-spectral radius of a connected $k$-uniform hypergraph $G$, we have
\begin{equation*}
	\rho_{\alpha}(G)=\max\{x^{\top}(\mathcal{A}_{\alpha}x)|x\in\mathbb{R}^{n}_{+},\Vert x\Vert_k=1\},
\end{equation*}
where $x$ is the $\alpha$-Perron vector of $G$. Let $x_{U}=\Pi_{v\in U}x_{v}$ for a subset $U\in V(G)$. Then for each $v\in V(G)$, we see that
\begin{equation*}
	\left( \mathcal{A}_{\alpha}(G)x\right)_{v}=\alpha d_{G}(v)x_{v}^{k-1}+(1-\alpha)\sum_{e\in E(G)}x_{e\backslash\{v\}}, 
\end{equation*}
and
\begin{equation*}
	\begin{split}
		x^{\top}\left( \mathcal{A}_{\alpha}(G)x\right)&=\alpha\sum_{v\in V(G)} d_{G}(v)x_{v}^{k}+(1-\alpha)k\sum_{e\in E(G)}x_{e}\\
		&=\sum_{e\in E(G)}\left( \alpha\sum_{v\in e}x_{v}^{k}+(1-\alpha)kx_{e}\right).
	\end{split}
\end{equation*}

Next, we introduce some operations for hypergraphs.
\begin{definition}[edge moving operation]\cite{LiShao and Qi}\label{def1}
	Let $r\geq1$ be a positive integer and $G$ be a $k$-uniform hypergraph with $u,v_1,\dots,v_r\in V(G)$ and $e_1,\dots,e_r\in E(G)$ such that $u\notin e_i$ and $v_i\in e_i$ for $1\leq i\leq r$. The vertices $v_1,\dots,v_r$ are not necessarily pairwise distinct. Let $e'_i=(e_i\backslash\{v_i\})\cup\{u\}$ for $1\leq i\leq r$ and $G'$ be the hypergraph with $V(G')=V(G)$ and $E(G')=(E(G)\backslash\{e_1,\dots,e_r\})\cup\{e'_1,\dots,e'_r\}$. Then we say that $G'$ is obtained from $G$ by moving edges $(e_1,\dots,e_r)$ from $(v_1,\dots,v_r)$ to $u$.
\end{definition}

The edge-releasing operation on supertrees is a special case of edge moving operation. 
\begin{definition}[edge-releasing operation]\cite{LiShao and Qi}\label{def2}
	Let $G$ be a $k$-uniform supertree with $u\in e\in E(G)$, where $e$ is a non-pendent edge. Let $e_1,\dots,e_r$ be all edges of $G$ adjacent to $e$ but not containing $u$. Suppose that $e_i\cap e=\{v_i\}$ for $1\leq i\leq r$. Let $G'$ be the hypergraph obtained from $G$ by moving edges ($e_1,\dots,e_r$) from ($v_1,\dots,v_r$) to $u$. Then we say that $G'$ is obtained from $G$ by an edge-releasing operation on $e$ at $u$. 
\end{definition}
	By Proposition 17 in \cite{LiShao and Qi}, we see that the hypergraph $G'$ obtained from a $k$-uniform supertree $G$ by edge-releasing a non-pendent edge $e$ of $G$ is also a supertree.
\begin{definition}[2-switching operation]\cite{Behrens,XiaoWan and Lu}\label{def3}
	Let $G$ be a $k$-uniform hypergraph with $e=\{u_1,\dots,u_k\}\in E(G)$ and $f=\{v_1,\dots,v_k\}\in E(G)$. Let $U_1=\{u_1,\dots,u_r\}$, $V_1=\{v_1,\dots,v_r\}$ for $1\leq r<k$, and let $e'=(e\backslash U_1)\cup V_1$, $f'=(f\backslash V_1)\cup U_1$. Let $G'$ be the hypergraph with $V(G')=V(G)$ and $E(G')=\left( E(G)\backslash\{e,f\}\right)\cup\{e',f'\} $. Then we say that $G'$ is obtained from $G$ by $e\xrightleftharpoons[v_1,\dots,v_r]{u_1,\dots,u_r}f$ or $e\xrightleftharpoons[V_1]{U_1}f$.
\end{definition}

Notice that the degrees of the vertices remain unchanged by applying a 2-switch operation to a hypergraph. In the following, we illustrate the effect of these operations on the $\alpha$-spectral radius of hypergraphs.

\begin{lemma}\label{moved}\cite{Guo and Zhou}
	Let $G'$ be the hypergraph obtained from a connected $k$-uniform hypergraph $G$ by moving edge $(e_1,\dots,e_r)$ from $(v_1,\dots,v_r)$ to $u$, where $r\geq1$. Let $x$ be the $\alpha$-Perron vector of $G$. If $x_u\geq\max\limits_{1\leq i\leq r}x_{v_i}$, then $\rho_{\alpha}(G')>\rho_{\alpha}(G)$.
\end{lemma}

\begin{lemma}\label{edgere}
	Let $G'$ be the supertree obtained from a $k$-uniform supertree $G$ by an edge-releasing operation on a non-pendent edge $e\in E(G)$ to a vertex $u\in e$. Then $\rho_{\alpha}(G')>\rho_{\alpha}(G)$.
\end{lemma}
\begin{proof}
	Let $x$ be the $\alpha$-Perron vector corresponding to $\rho_{\alpha}(G)$. Take a non-pendent edge $e\in E(G)$ and a vertex $v\in e$ such that $x_v=\max\limits_{v'\in e}\{x_{v'}\}$. Let $G''$ be the supertree obtained from a $k$-uniform supertree $G$ by an edge-releasing operation on $e$ to $v$. By Definition \ref{def2} and Lemma \ref{moved}, we have $\rho_{\alpha}(G'')>\rho_{\alpha}(G)$. Note that supertrees $G'$ and $G''$ are isomorphic, thus $\rho_{\alpha}(G')=\rho_{\alpha}(G'')>\rho_{\alpha}(G)$. This completes the proof.
\end{proof}

\begin{lemma}\label{2sw}\cite{DouWang}
	Let $G$ be a connected $k$-uniform hypergraph with $e=\{u_1,\dots,u_k\}\in E(G)$ and $f=\{v_1,\dots,v_k\}\in E(G)$. Let $U_1=\{u_1,\dots,u_r\}$, $V_1=\{v_1,\dots,v_r\}$ for $1\leq r<k$ and let $U_2=e\backslash U_1$, $V_2=f\backslash V_1$. Suppose that $e'=U_1\cup V_2$ and $f'=V_1\cup U_2$ are $k$-subsets of $V(G)$ and not in $E(G)$. Let $G'$ be obtained from $G$ by $e\xrightleftharpoons[V_1]{U_1}f$ and $x$ be the $\alpha$-Perron vector of $G$. If $x_{U_1}\geq x_{V_1}$ and $x_{U_2}\leq x_{V_2}$, then $\rho_{\alpha}(G')\geq\rho_{\alpha}(G)$. Equality holds if and only if $x_{U_1}=x_{V_1}$ and $x_{U_2}=x_{V_2}$.
\end{lemma}

\section{The $\alpha$-spectral radius of supertrees with a given independence number}
In this section, we mainly determine the unique supertree with the maximum $\alpha$-spectral radius among all $k$-uniform supertrees with a given independence number. For a $k$-uniform supertree $G$ with $m$ edges and independence number $\beta$, it follows that $k\beta\geq|V(G)|=n=m(k-1)+1$, implying $\beta\geq\frac{m(k-1)+1}{k}$.   
\begin{lemma}\label{le3_1}\cite{GuoanZhou srh}
	There exists a maximum independence set $S$ of a supertree $G$ such that there is a pendent vertex in $S$ for each pendent edge of $G$.
\end{lemma}
	A $k$-uniform hyperstar with $m$ edges is written as $S_{m,k}$. In \cite{GuoanZhou srh}, Guo and Zhou characterized the extremal $k$-uniform supertrees with maximum spectral radius and independent number $\beta$, shown in the following definition.
\begin{definition}\cite{GuoanZhou srh}
	Let $T_{m,k,\beta}$ be a $k$-uniform supertree obtained from a hyperstar $S_{(k-1)\beta-(k-2)m,k}$ by attaching a pendent edge at each pendent vertex of $m-\beta$ edges in $S_{(k-1)\beta-(k-2)m,k}$ , where $m,k,\beta$ are three positive integers satisfying $k\geq2$ and $\lceil\frac{m(k-1)+1}{k}\rceil\leq\beta\leq m$.
\end{definition}
\begin{figure}[htbp]
	\centering
	\subfigure[Supertree $T_{8,3,6}$]{
		\includegraphics[width=6.8cm]{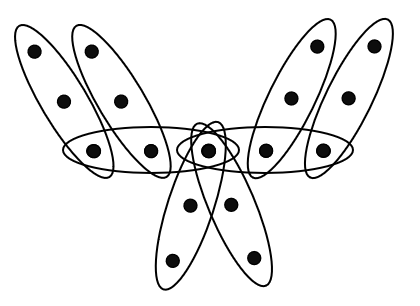}
	}
	\quad
	\subfigure[Supertree $T_{5,4,4}$]{
		\includegraphics[width=5.8cm]{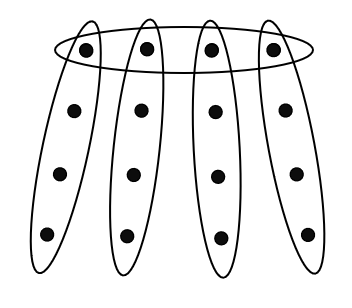}
	}
	\caption{Two supertrees in $T_{m,k,\beta}$.}
	\label{fig1}
\end{figure}

\begin{theorem}
	For $\lceil\frac{m(k-1)+1}{k}\rceil\leq\beta\leq m$, $T_{m,k,\beta}$ is the unique supertree with maximum $\alpha$-spectral radius among all $k$-uniform supertrees with $m$ edges and independence number $\beta$.
\end{theorem}
\begin{proof}
	Assume that $G$ is a supertree with maximum $\alpha$-spectral radius among all $k$-uniform supertrees with $m$ edges and independence number $\beta$. Then we prove that $G$ must be in $T_{m,k,\beta}$.
	
	Let $x$ be the $\alpha$-Perron vector of $G$. Suppose that $S$ is a maximum independent set containing pendent vertices as many as possible. By Lemma \ref{le3_1}, we see that $S$ contains a pendent vertex of each pendent edge in $E(G)$. Take $\Gamma$ to be the $k$-uniform supertree obtained from $G$ by deleting all pendent vertices in all pendent edges of $G$. Then we consider the following cases.
	
	\textbf{Case 1.} There is no vertex $v\in V(\Gamma)$ such that $d_{\Gamma}(v)\geq2$, i.e., $\Gamma$ only contains an isolate vertex or $\Gamma\cong S_{1,k}$.
	
	\textbf{Case 1.1.} $\Gamma$ only has an isolated vertex. Easily, we have $\beta(G)=m$ and $G\cong S_{m,k}\cong T_{m,k,m}$.
	
	\textbf{Case 1.2.} $\Gamma\cong S_{1,k}$. There exist some vertices $\{v_1,\dots,v_r\}\subset V(\Gamma)$ with $2\leq r\leq k$ such that $d_{G}(v_i)\geq2$ for $1\leq i\leq r$. Assume that there are no less than one pendent vertex of $G$ in $\Gamma$, then we have $\beta(G)=m$.  Without loss of generality, assume that $x_{v_1}\geq x_{v_2}$. Let $G'$ be the supertree obtained from $G$ by moving a pendent edge containing $v_2$ from $v_2$ to $v_1$. We can see that $G'$ is also a $k$-uniform supertree with $m$ edges and $\beta(G')=\beta(G)=m$. By Lemma \ref*{moved}, we have $\rho_{\alpha}(G')>\rho_{\alpha}(G)$, a contradiction. Hence, there is no pendent vertex of $G$ in $\Gamma$, implying $\beta(G)=m-1$ and $r=k$.
	
	Let $e=\{v_1,\dots,v_k\}$ be the single edge of $\Gamma$. It follows that $d_{G}(v_i)\geq2$ for $1\leq i \leq k$. Assume that there are some vertices $\{v_1,\dots,v_t\}\subset e$ with $2\leq t\leq k$ such that $d_G(v_i)\geq3$ for $1\leq i\leq t$ and $x_{v_1}\geq x_{v_2}$. Let $G'$ be the supertree obtained from $G$ by moving a pendent edge incident with $v_2$ from $v_2$ to $v_1$. It's obvious that $G'$ is a $k$-uniform supertree with $m$ edges and $\beta(G')=\beta(G)=m-1$. Referring to Lemma \ref{moved}, we see that $\rho_{\alpha}(G')>\rho_{\alpha}(G)$, a contradiction. Thus, $t=0,1$. There is at most one vertex of $\Gamma$ with degree at least $3$ in $G$, implying $G\cong T_{m,k,m-1}$.
	
	\textbf{Case 2.} There exist some vertices $\{v_1,\dots,v_l\}\subset V(\Gamma)$ such that $d_{\Gamma}(v_i)\geq2$ for $1\leq i\leq l$.
	
	\begin{claim}\label{ca1}
		For any $v\in\{v_1,\dots,v_l\}$, $v$ must be contained in some pendent edge of $G$.
	\end{claim}  
	\begin{proof}[Proof of Claim \ref{ca1}]
		Assume that there exists a vertex $v\in\{v_1,\dots,v_l\}$ such that $v$ is not in any pendent edge of $G$. Then we consider the following two cases.
	
		\textbf{(i)} $v\in S$. For an edge $e\in E(G)$ with $v\in e$, there is no pendent vertex in $e$. Otherwise, we can replace $v$ by a pendent vertex in $e$ to get a maximum independent set $S'$ containing more pendent vertices than $S$, a contradiction.
		
		Suppose that there are two edges $e'$ and $e''$ containing $v$ with $v'\in e'$ and $v''\in e''$ such that $v$ is the unique neighbor of $v'$ ($v''$) in $S$. Then we can replace $v$ by $v',v''$ to obtain an independent set $S'$ such that $|S'|>|S|$, a contradiction. Therefore, there exists at most one edge $e'$ with $\{v,v'\}\subset e'$ such that $v$ is  the unique neighbor of $v'$ in $S$. 
		It's clear that $d_{G}(v')\geq2$ and $v'\notin S$.  We may assume that $x_{v'}\geq x_{u}$ for any $u\in e'\backslash\{v\}$. Let $G'$ be the supertree obtained from $G$ by moving all edges containing $u$ (except $e'$) from $u$ to $v'$. One can check that $G'$ is a $k$-uniform supertree with $m$ edges and $\beta(G')=\beta(G)$. By Lemma \ref{moved}, we have $\rho_{\alpha}(G')>\rho_{\alpha}(G)$, a contradiction. Thus a neighbor of $v$ is adjacent to another vertex (different from $v$) in $S$. 
		
		We take an arbitrary edge $e\in E(G)$ containing $v$. Denote $e=\{v,u_1,\dots,u_{k-1}\}$. Easily, we see that $d_{G}(u_i)\geq 2$, $u_i\notin S$, and $u_i$ has another neighbor (different from $v$) in $S$, for $1\leq i\leq k-1$. We may assume that $x_{u_1}=\max_{1\leq i\leq k-2}x_{u_i}$. Let $G''$ be the supertree obtained from $G$ by moving all edges containing $u_i$ (except $e$) from $u_i$ to $u_1$, where $2\leq i\leq k-1$. It can be checked that $G''$ is a $k$-uniform supertree with $m$ edges and $\beta(G'')=\beta(G)$. By Lemma \ref{moved}, we have $\rho_{\alpha}(G'')>\rho_{\alpha}(G)$, a contradiction.
		
		\textbf{(ii)} $v\notin S$. Suppose that $u\in S$ is a neighbor of $v$ with $d_{G}(u)\geq2$. Similarly, we conclude that there is no pendent vertex in each edge containing $u$ and a neighbor of $u$ is adjacent to another vertex (different from $u$) in $S$. Let $f=\{u,w_1,\dots,w_{k-1}\}$. One can see that $d_{G}(w_i)\geq2$ and $w_i\notin S$ for $1\leq i\leq k-1$. We may assume that $x_{w_1}=\max_{1\leq i\leq k-1}x_{w_i}$. Let $G'''$ be the supertree obtained from $G$ by moving all edges incident with $w_i$ (except $f$) from $w_i$ to $w_1$ for $2\leq i\leq k-1$. Obviously, $G'''$ is a $k$-uniform supertree with $m$ edges and $\beta(G''')=\beta(G)$. From Lemma \ref{moved}, it follows that $\rho_{\alpha}(G''')>\rho_{\alpha}(G)$, a contradiction.
		
		Suppose that there exists a pendent vertex $u'\in S$ such that $u'$ and $v$ are adjacent. Let $f'=\{v,u',y_1,\dots,y_{k-2}\}$. Note that $f'$ is a non-pendent edge. Let $G''''$ be the $k$-uniform supertree obtained from $G$ by an edge-releasing operation on $f'$ at $v$. By Lemma \ref{edgere}, we have $\rho_{\alpha}(G'''')>\rho_{\alpha}(G)$, a contradiction.
		
		This proves the Claim \ref{ca1}.
%
	\end{proof}

	By the above discussion, we take two vertices $v_1,v_2\in V(\Gamma)$ satisfying $d_{\Gamma}(v_i)\geq2$ for $i=1,2$ such that these two vertices are contained in some pendent edge of $G$. By the choice of $S$, we have $v_i\notin S$ for $i=1,2$. Without loss of generality, we suppose that $x_{v_1}\geq x_{v_2}$. Let $G'$ be the supertree obtained from $G$ by moving a non-pendent edge containing $v_2$ from $v_2$ to $v_1$. It can be checked that $G'$ is a $k$-uniform supertree with $m$ edges and $\beta(G')=\beta(G)$. By Lemma \ref{moved}, we have $\rho_{\alpha}(G')>\rho_{\alpha}(G)$, a contradiction. Hence, $l=1$. There is only one vertex $v_1\in V(\Gamma)$ such that $d_{\Gamma}(v_1)\geq 2$. It's easy to see that $\Gamma$ is a $k$-uniform hyperstar with the center $v_1$ and $G$ is the supertree obtained from $\Gamma$ by attaching some pendent edges at vertices of $\Gamma$.
	
	Take an edge $e\in E(\Gamma)$. There is at least one pendent edge at some vertex in $e\backslash\{v_1\}$. We may assume that there are $a$ pendent edges at $b$ vertices of $e\backslash\{v_1\}$ in $G$, where $a\leq b$ and $1\leq b\leq k-1$. Suppose that $b\leq k-2$. Since $v_1\notin S$, we see that there exists a pendent vertex of $e$ in $S$. Let $G''$ be the supertree obtained from $G$ by an edge-releasing operation on $e$ at $v_1$. One can check that $G''$ is a $k$-uniform supertree with $m$ edges and $\beta(G'')=\beta(G)$. Combining with Lemma \ref{edgere}, we obtain $\rho_{\alpha}(G'')>\rho_{\alpha}(G)$, a contradiction. Hence, $b=k-1$.
	
	Now, it's clear that the number of non-pendent vertices of an edge of $\Gamma$ in $G$ is either 1 or $k$. Suppose that there exists a vertex $u\in V(\Gamma)\backslash\{v_1\}$ satisfying $d_{G}(u)\geq3$. If $x_{u}\geq x_{v_1}$, let $G'''$ be the supertree obtained from $G$ by moving a pendent edge at $v_1$ from $v_1$ to $u$, and if $x_{u}<x_{v_1}$, let $G''''$ be the supertree obtained from $G$ by moving a pendent edge at $u$ from $u$ to $v_1$. Obviously, $G'''$ and $G''''$ are both $k$-uniform supertree with $m$ edges and $\beta(G''')=\beta(G'''')=\beta(G)$. By Lemma \ref{moved}, we have $\rho_{\alpha}(G''')>\rho_{\alpha}(G)$ and $\rho_{\alpha}(G'''')>\rho_{\alpha}(G)$, a contradiction. Therefore, for any $u\in V(\Gamma)\backslash\{v_1\}$, $d_{G}(u)=2$, which implies that $G\cong T_{m,k,\beta}$.
	
	As can be seen, $T_{m,k,\beta}$ is the unique supertree with maximum $\alpha$-spectral radius among all $k$-uniform supertrees with $m$ edges and independence number $\beta$. This completes the proof.
\end{proof}

\section{The $\alpha$-spectral radius of supertrees with given degree sequences}
	In this section, we characterize the unique supertree with maximum $\alpha$-spectral radius among all $k$-uniform supertrees with given degree sequences. We first prove a necessary lemma.
	
	\begin{definition}\cite{XiaoWan and Lu}
		Let $\pi=(d_0,d_1,\dots,d_{n-1})$ be a non-increasing degree sequence. 
		The graph family $\mathbb{T}_{\pi}$ is the set containing all $k$-uniform supertrees with $\pi$ as their degree sequence.
	\end{definition}

	\begin{lemma}\label{le4_2}
		Let $G$ be a $k$-uniform supertree with the maximum $\alpha$-spectral radius in $\mathbb{T}_{\pi}$ and let $x$ be the $\alpha$-Perron vector of $G$. For $u,v\in V(G)$, if $d_{G}(u)>d_{G}(v)$, then $x_{u}>x_{v}$.
	\end{lemma}
	\begin{proof}
		By way of contradiction assume that $x_{u}\leq x_{v}$. Notice that $G$ is a supertree, there exists a unique path $u=v_1,e_1,\dots,e_l,v_{l+1}=v$ from $u$ to $v$. Let  $s=d_{G}(u)-d_{G}(v)$, there exist some edges $f_1,\dots,f_{s}$ (except $e_1,\dots,e_l$) incident with $u$ but not incident with $v$. Let $G'$ be the supertree obtained from $G$ by moving edges $(f_1,\dots,f_s)$ from $u$ to $v$. It can be checked that $G'$ is also a $k$-uniform supertree in $\mathbb{T}_{\pi}$. By Lemma \ref{moved}, we get $\rho_{\alpha}(G')>\rho_{\alpha}(G)$, a contradiction.
	\end{proof}

	\begin{corollary}\label{cor1}
		Let $G$ be a $k$-uniform supertree with the maximum $\alpha$-spectral radius in $\mathbb{T}_{\pi}$ and let $x$ be the $\alpha$-Perron vector of $G$. For $u,v\in V(G)$, if $x_{u}\geq x_{v}$, then $d_{G}(u)\geq d_{G}(v)$. Moreover, if $x_u=x_v$, $d_{G}(u)=d_{G}(v)$.
	\end{corollary}
	
	Next, we introduce the definition of breadth-first-search ordering (BFS-ordering for short) and construct a special supertree with a BFS-ordering of its vertices in $\mathbb{T}_{\pi}$.
	Let $G$ be a $k$-uniform supertree with root $v_0$. For $v\in V(G)$, $d_{G}(v,v_0)$ is called the height of $v$ in $G$, written as $h(v)$. We also say that $v$ is in layer $h(v)$ of $G$.
	\begin{definition}[BFS-ordering]\label{def4_4}\cite{XiaoWan and Lu}
		Let $G$ be a $k$-uniform supertree with root $v_0$. A well-ordering $\prec$ of the vertices of $G$ is called a BFS-ordering if all the following hold (i)-(iv) for all vertices:
		\begin{itemize}
			\item[(\romannumeral1)] $u\prec v$ implies $h(u)\leq h(v)$;
			\item[(\romannumeral2)] $u\prec v$ implies $d_{G}(u)\geq d_{G}(v)$;
			\item[(\romannumeral3)] For $\{u,u'\}\subset e\in E(G)$ and $\{v,v'\}\subset f\in E(G)$ with $h(u)=h(u')+1$ and $h(v)=h(v')+1$, if $u\prec v$, then $u'\prec v'$;
			\item[(\romannumeral4)] Suppose that $v_1\prec v_2\prec\dots\prec v_k$ for every edge $e=\{v_1,v_2,\dots,v_k\}\in E(G)$, then there exists no vertex $u\in V(G)\backslash e$ such that $v_i\prec u\prec v_{i+1}$, where $2\leq i\leq k-1$. 
		\end{itemize}
%
%
%
	\end{definition}

	A $k$-uniform supertree whose vertices have a BFS-ordering is called a BFS-supertree. In \cite{XiaoWan and Lu}, Xiao et al. construct a BFS-supertree $G_{\pi}$ in $\mathbb{T}_{\pi}$ by constructing layers as follows:
	\begin{itemize}
		\item[(1)] Select a vertex $v_{0,1,1}$ as a root and begin with vertex $v_{0,1,1}$ in layer $0$. Set $c_1=d_0$.
		\item[(2)] Select $(k-1)c_1$ vertices:
		\begin{equation*}
			v_{1,1,1},\dots,v_{1,1,k-1},v_{1,2,1},\dots,v_{1,2,k-1},\dots,v_{1,c_1,1},\dots,v_{1,c_1,k-1}
		\end{equation*}
	in layer 1 such that $\{v_{0,1,1},v_{1,t,1},\dots,v_{1,t,k-1}\}\in E(G)$ for $1\leq t\leq d_{0}$. Set $c_2=\sum_{p=1}^{(k-1)c_1}d_{p}-(k-1)c_1$. 
		\item[(3)] Select $(k-1)c_2$ vertices:
		\begin{equation*}
			v_{2,1,1},\dots,v_{2,1,k-1},v_{2,2,1},\dots,v_{2,2,k-1},\dots,v_{2,c_2,1},\dots,v_{2,c_2,k-1}
		\end{equation*}
	in layer 2 such that $\{v_{1,i,j},v_{2,t,1},\dots,v_{2,t,k-1}\}\in E(G)$ for
	\begin{equation*}
		\sum_{p=1}^{(i-1)(k-1)+j-1}d_p-[(i-1)(k-1)+j-1]+1\leq t\leq \sum_{p=1}^{(i-1)(k-1)+j}d_p-[(i-1)(k-1)+j],
	\end{equation*}
	where $1\leq i\leq c_1$ and $1\leq j\leq k-1$.
	Set $c_3=\sum_{p=(k-1)c_1+1}^{(k+1)(c_1+c_2)}d_p-(k-1)c_2$.
	\item[(4)] Assume that all vertices in layer $a$ have been constructed. These vertices are
	\begin{equation*}
		v_{a,1,1},\dots,v_{a,1,k-1},v_{a,2,1},\dots,v_{a,2,k-1},\dots,v_{a,c_a,1},\dots,v_{a,c_a,k-1}.
	\end{equation*}
	Set $c_{a+1}=\sum_{p=(k-1)(c_1+\cdots+c_{a-1})+1}^{(k-1)(c_1+\cdots+c_{a})}d_p-(k-1)c_a$. Then we construct layer $a+1$. Select $(k-1)c_{a+1}$ vertices
	\begin{equation*}
		v_{a+1,1,1},\dots,v_{a+1,1,k-1},v_{a+1,2,1},\dots,v_{a+1,2,k-1},\dots,v_{a+1,c_{a+1},1},\dots,v_{a+1,c_{a+1},k-1}
	\end{equation*}
	in layer $a+1$ such that $\{v_{a,i,j},v_{a+1,t,1},\dots,v_{a+1,t,k-1}\}\in E(G)$ for 
	\begin{equation*}
		\begin{split}
		&	\sum_{p=(k-1)(c_1+\cdots+c_{a-1})+1}^{(k-1)(c_1+\cdots+c_{a-1})+(i-1)(k-1)+j-1}d_p-[(i-1)(k-1)+j-1]+1\\
		&\leq t\\
		&\leq\sum_{p=(k-1)(c_1+\cdots+c_{a-1})+1}^{(k-1)(c_1+\cdots+c_{a-1})+(i-1)(k-1)+j}d_p-[(i-1)(k-1)+j]
		\end{split}.
	\end{equation*}
	where $1\leq i\leq c_a$ and $1\leq j\leq k-1$.
	Set $c_{a+2}=\sum_{p=(k-1)(c_1+\cdots+c_a)+1}^{(k-1)(c_1+\cdots+c_{a+1})}d_p-(k-1)c_{a+1}$.
	\item[(5)] We stop in some layer if all the vertices are in the layers.
	\end{itemize}

	From Proposition 3.1 in \cite{XiaoWan and Lu}, we see that there exists a $k$-uniform supertree $G_{\pi}$ in $\mathbb{T}_{\pi}$ having BFS-ordering of vertices. In addition, any two $k$-uniform supertrees with the same degree sequence and having BFS-ordering of their vertices are isomorphic.
	
	\begin{theorem}
		For a given non-increasing degree sequence $\pi=\{d_0,d_1,\dots,d_{n-1}\}$ of some $k$-uniform supertrees, if $G$ attains the maximum $\alpha$-spectral radius in $\mathbb{T}_{\pi}$, then $G$ is a BFS-supertree.
	\end{theorem}
	\begin{proof}
		Let $x$ be the $\alpha$-Perron vector of $G$. Note that $d_0\geq d_1\geq\dots\geq d_{n-1}$ are degree of vertices in $G$. Correspondingly, we denote $V(G)=\{v_0,v_1,\dots,v_{n-1}\}$ in here. By Lemma \ref{le4_2}, we have $x_{v_0}=\max_{v\in V(G)}x_v$.
		
		Let $V_i=\{v:h(v)=i\}$ for $i=0,1,\dots,l+1$, $y_i=|V_i|$, and $z_i=\sum_{v\in V_i}d_{G}(v)$ for $i=0,1,\dots,l+1$. Then we relabel all vertices of $V(G)$ as follows \cite{XiaoWan and Lu}:
		\begin{itemize}
			\item[(1)] Relabel $v_0$ by $v_{0,1,1}$ as the root of the supertree $G$.
			\item[(2)] Relabel $(k-1)z_0$ vertices in $V(G)$ are adjacent to $v_{0,1,1}$ as
			\begin{equation*}
				v_{1,1,1},\dots,v_{1,1,k-1},v_{1,2,1},\dots,v_{1,2,k-1},\dots,v_{1,z_0,1},\dots,v_{1,z_0,k-1}.
			\end{equation*}
		These vertices satisfy following three conditions:
			\begin{itemize}
				\item[(a.1)] $e_{1,i}=\{v_{0,1,1},v_{1,i,1},\dots,v_{1,i,k-1}\}\in E(G)$ for $1\leq i\leq z_0$.
				\item[(b.1)] $\prod_{p=1}^{k-1}x_{v_{1,i,p}}\geq\prod_{p=1}^{k-1}x_{v_{1,j,p}}$ for $1\leq i<j\leq z_0$.
				\item[(c.1)] $x_{v_{1,i,p}}\geq x_{v_{1,i,q}}$ for $1\leq i\leq z_0$ and $p<q$.
			\end{itemize}
		\item[(3)] Assume that $(k-1)(z_{r-1}-y_{r-1})$ vertices in layer $r$ are relabeled, where $1\leq r\leq l$. Now, we relabel $(k-1)(z_r-y_r)$ vertices in layer $r+1$ as
		\begin{equation*}
			v_{r+1,1,1},\dots,v_{r+1,1,k-1},v_{r+1,2,1},\dots,v_{r+1,2,k-1},\dots,v_{r+1,z_r-y_r,1},\dots,v_{r+1,z_r-y_r,k-1}.
		\end{equation*}
		These vertices satisfy following conditions:
		\begin{itemize}
			\item[(a.2)] $\{v_{r,i_0,p},v_{r+1,i,1},\dots,v_{r+1,i,k-1}\}\in E(G)$ for $1\leq i\leq z_r-y_r$, where $1\leq i_0\leq z_{r-1}-y_{r-1}$ and $1\leq p\leq k-1$.
			\item[(b.2)] $\prod_{p=1}^{k-1}x_{v_{r+1,i,p}}\geq \prod_{p=1}^{k-1} x_{v_{r+1,j,p}}$ for $1\leq i<j\leq z_r-y_r$.
			\item[(c.2)] For two edges $\{v_{r,i_0,p},v_{r+1,i,1},\dots,v_{r+1,i,k-1}\}$ and $\{v_{r,j_0,q},v_{r+1,j,1},\dots,v_{r+1,j,k-1}\}$ in $E(G)$, if $\prod_{c=1}^{k-1}x_{v_{r+1,i,c}}=\prod_{c=1}^{k-1}x_{v_{r+1,j,c}}$ where $1\leq i< j\leq z_r-y_r$, then $i_0<j_0$ or $i_0=j_0$, $p\leq q$.
			\item[(d.2)] $x_{v_{r+1,i,p}}\geq x_{v_{r+1,i,q}}$ for $1\leq i\leq z_r-y_r$ and $1\leq p<q\leq k-1$.
		\end{itemize}
		\item[(4)] We stop in some layer until all vertices of $V(G)$ have been relabeled.
		\end{itemize}
	
		Then we propose a well-ordering of vertices in $G$, that is
		\begin{equation}\label{wellorder}
			v_{s,i,p}\prec v_{r,j,q} \mbox{ for }s<t \mbox{ or }s=t,i<j \mbox{ or } s=t,i=j,p<q.
		\end{equation}
	\begin{claim}\label{ca2}
		For all vertices in same layer $s$ with $s=0,1,\dots,l+1$, the inequality \begin{equation*}
			x_{v_{s,i,q}}\geq x_{v_{s,j,q}}
		\end{equation*} 
		holds for $i<j$ or $i=j,p<q$.
	\end{claim}
	\begin{proof}[Proof of Claim \ref{ca2}]
		Obviously, the Claim \ref{ca2} holds for s=0. By mathematical induction, we assume that the Claim \ref{ca2} holds for $s=r$. For $s=r+1$, by condition (d.2), we see that $x_{v_{r+1},i,p}\geq x_{v_{r+1},j,q}$ for $i=j,p<q$.
		
		Suppose that there exist two vertices $v_{r+1,i,p_0}$, $v_{r+1,j,q_0}\in V_{r+1}$ such that $x_{v_{r+1,i,p_0}}<x_{v_{r+1,j,q_0}}$, where $i<j$, $1\leq p_0,q_0\leq k-1$. Take two edges $e_{r+1,i}=\{v_{r,i_1,p_1},v_{r+1,i,1},\dots,v_{r+1,i,k-1}\}$ and $e_{r+1,j}=\{v_{r,j_1,q_1},v_{r+1,j,1},\dots,v_{r+1,j,k-1}\}$ in $E(G)$. Then we consider the following three cases.
		
		\textbf{Case 1.} $v_{r,i_1,p_1}=v_{r,j_1,q_1}$.
		
		Notice that $x_{v_{r+1,i,p_0}}<x_{v_{r+1,j,q_0}}$ and $\prod_{p=1}^{k-1}x_{v_{r+1,i,p}}\geq \prod_{q=1}^{k-1}x_{v_{r+1,j,q}}$, then we obtain that $\prod_{p=1,p\ne p_0}^{k-1}x_{v_{r+1,i,p}}>\prod_{q=1,q\ne q_0}^{k-1}x_{v_{r+1,j,q}}$. Let $G'$ be the supertree obtained from $G$ by $e_{r+1,i}\xrightleftharpoons[v_{r+1,j,q_0}]{v_{r+1,i,p_0}}e_{r+1,j}$. From Lemma \ref{2sw}, we have $\rho_{\alpha}(G')>\rho_{\alpha}(G)$, a contradiction.
		
		\textbf{Case 2.} $v_{r,i_1,p_1}\prec v_{r,j_1,q_1}$.
		
		It can be seen that $i_1<j_1$ or $i_1=j_1$, $p_1<q_1$, then we have $x_{v_{r,i_1,p_1}}\geq x_{v_{r,i_2,p_2}}$.
		Note that $x_{v_{r+1,i,p_0}}<x_{v_{r+1,j,q_0}}$ and $\prod_{p=1}^{k-1}x_{v_{r+1,i,p}}\geq \prod_{q=1}^{k-1}x_{v_{r+1,j,q}}$, then we get $x_{e_{r+1,i}\backslash\{v_{r+1,i,p_0}\}}>x_{e_{r+1,j}\backslash\{v_{r+1,j,q_0}\}}$. Let $G''$ be the supertree obtained from $G$ by $e_{r+1,i}\xrightleftharpoons[v_{r+1,j,q_0}]{v_{r+1,i,p_0}}e_{r+1,j}$. By Lemma \ref{2sw}, we have $\rho_{\alpha}(G'')>\rho_{\alpha}(G)$, a contradiction. 
		
		\textbf{Case 3.} $v_{r,j_1,q_1}\prec v_{r,i_1,p_1}$.
		
		We see that $j_1<i_1$ or $j_1=i_1$, $q_1<p_1$, which implies that $x_{v_{r,i_1,p_1}}\leq x_{v_{r,j_1,q_1}}$. Let $U_i=e_{r+1,i}\backslash\{v_{r,i_1,p_1},v_{r+1,i,p_0}\}$ and $U_j=e_{r+1,j}\backslash\{v_{r,j_1,q_1},v_{r+1,j,q_0}\}$.
		Note that $x_{v_{r+1,i,p_0}}<x_{v_{r+1,j,q_0}}$ and  $\prod_{p=1}^{k-1}x_{v_{r+1,i,p}}\geq \prod_{q=1}^{k-1}x_{v_{r+1,j,q}}$, then we get $x_{U_i}>x_{U_j}$. Let $G'''$ be the supertree obtained from $G$ by $e_{r+1,i}\xrightleftharpoons[U_j]{U_i}e_{r+1,j}$. By Lemma \ref{2sw}, we have $\rho_{\alpha}(G''')>\rho_{\alpha}(G)$, a contradiction.
		
		This prove the Claim \ref{ca2}.
	\end{proof}

	\begin{claim}\label{ca3}
	    For vertices in layers $s$ and $s+1$ with $s=0,1,\dots,l$, it follows that 
	    	$x_{v_{s,z_{s-1}-y_{s-1},k-1}}\geq x_{v_{s+1,1,1}}$
    for $s\geq1$, and
    		$x_{v_{s,1,1}}\geq x_{v_{s+1,1,1}}$
    for $s=0$.
    \end{claim}
	\begin{proof}[Proof of Claim \ref{ca3}]
		Since $x_{v_{0,1,1}}=\max_{v\in V(G)}x_v$, it's easy to see the Claim  \ref{ca3} holds for $s=0$. By mathematical induction, we assume that the Claim \ref{ca3} holds for $s=r$. For $s=r+1$, we suppose that $x_{v_{r+1,z_{r}-y_{r},k-1}}< x_{v_{r+2,1,1}}$ and take two edges $e_{r+1,z_r-y_r}=\{v_{r,i_2,p_2},v_{r+1,z_r-y_r,1},\dots,v_{r+1,z_r-y_r,k-1}\}$, $e_{r+2,1}=\{v_{r+1,j_2,q_2},v_{r+2,1,1},\dots,v_{r+2,1,k-1}\}\in E(G)$. Then we consider the following two cases.
		
		\textbf{Case 1.} $e_{r+1,z_r-y_r}\cap e_{r+2,1}=\emptyset$.
		
		From Claim \ref{ca2}, we see that $x_{v_{r,i_2,p_2}}\geq x_{v_{r,z_{r-1}-y_{r-1},k-1}}\geq x_{v_{r+1,1,1}}\geq x_{v_{r+1,j_2,q_2}}$.
		
		If  $\prod_{p=1}^{k-2}x_{v_{r+1,z_r-y_r,p}}\leq\prod_{q=1}^{k-1}x_{v_{r+2,1,q}}$, we see that $x_{v_{r,i_2,p_2}}\geq x_{v_{r+1,j_2,q_2}}$ and $x_{e_{r+1,z_r-y_r}\backslash\{v_{r,i_2,p_2}\}}<x_{e_{r+2,1}\backslash\{v_{r+1,j_2,q_2}\}}$. Let $G'$ be the supertree obtained from $G$ by $e_{r+1,z_r-y_r}\xrightleftharpoons[v_{r+1,j_2,q_2}]{v_{r,i_2,p_2}}e_{r+2,1}$. By Lemma \ref{2sw}, we have $\rho_{\alpha}(G')>\rho_{\alpha}(G)$, a contradiction.
		
		If $\prod_{p=1}^{k-2}x_{v_{r+1,z_r-y_r,p}}>\prod_{q=1}^{k-1}x_{v_{r+2,1,q}}$, we get $x_{e_{r+1,z_r-y_r}\backslash\{v_{r+1,z_r-y_r,k-1}\}}>x_{e_{r+2,1}\backslash\{v_{r+2,1,1}\}}$ and $x_{v_{r+1,z_{r}-y_{r},k-1}}< x_{v_{r+2,1,1}}$. Let $G''$ be the supertree obtained from $G$ by $e_{r+1,z_r-y_r}\xrightleftharpoons[v_{r+2,1,1}]{v_{r+1,z_{r}-y_{r},k-1}}e_{r+2,1}$. From Lemma \ref{2sw}, we have $\rho_{\alpha}(G'')>\rho_{\alpha}(G)$, a contradiction.
		
		\textbf{Case 2.} $e_{r+1,z_r-y_r}\cap e_{r+2,1}=\{v_{r+1,z_r-y_r,c}\}$, where $1\leq c\leq k-1$.
		
		If $z_r-y_r=1$, we may take a pendent vertex $v_{r,i_3,p_3}$ in layer $r$. Notice that $x_{v_{r,i_3,p_3}}\geq x_{v_{r,z_{r-1}-y_{r-1},k-1}}\geq x_{r+1,1,1}\geq x_{v_{r+1,z_r-y_r,c}}$. Let $G'$ be the supertree obtained from $G$ by moving all edges containing $v_{r+1,z_r-y_r,c}$ (except $e_{r+1,z_r-y_r}$) from $v_{r+1,z_r-y_r,c}$ to $v_{r,i_3,p_3}$. Since $v_{r,i_3,p_3}$ is a pendent vertex, $G'$ is also a $k$-uniform supertree with degree sequence $\pi$. By Lemma \ref{moved}, we get $\rho_{\alpha}(G')>\rho_{\alpha}(G)$, a contradiction.
		
		If $z_r-y_r>1$ and $d_{G}(v_{r+1,1,1})=1$, we have $x_{v_{r+1,1,1}}\geq x_{v_{r+1,z_r-y_r,c}}$. Let $G''$ be the supertree obtained from $G$ by moving all edges containing $v_{r+1,z_r-y_r,c}$ (except $e_{r+1,z_r-y_r}$) from  $v_{r+1,z_r-y_r,c}$ to $v_{r+1,1,1}$. Similarly, one can check that $G''$ is a $k$-uniform supertree with degree sequence $\pi$. By Lemma \ref{moved}, we have $\rho_{\alpha}(G'')>\rho_{\alpha}(G)$, a contradiction.
		
		If $z_r-y_r>1$ and $d_{G}(v_{r+1,1,1})\geq2$, we have $x_{v_{r+1,1,1}}\geq x_{v_{r+1,z_r-y_r,c}}$. Without loss of generality, we denote $e_{r+2,2}=\{v_{r+1,1,1},v_{r+2,2,1},\dots,v_{r+2,2,k-1}\}$. See condition (c.2), we get $\prod_{p=1}^{k-1}x_{v_{r+2,1,p}}>\prod_{p=1}^{k-1}x_{v_{r+2,2,p}}$. Let $G'''$ be the supertree obtained from $G$ by $e_{r+2,1}\xrightleftharpoons[v_{r+1,1,1}]{v_{r+1,z_r-y_r,c}}e_{r+2,2}$. From Lemma \ref{2sw}, we have $\rho_{\alpha}(G''')>\rho_{\alpha}(G)$, a contradiction.
		
		This proves the Claim \ref{ca3}.
	\end{proof}

	By Claim \ref{ca2}, Claim \ref{ca3} and Corollary \ref{cor1}, we have the following relations
	\begin{equation*}
		d_{G}(v_{s,i,p})\geq d_{G}(v_{s,j,q})
	\end{equation*}
	holds for $i<j$ or $i=j,p<q$, and 
	\begin{equation*}
		d_{G}(v_{s,z_{s-1}-y_{s-1},k-1})\geq d_{G}(v_{s+1,1,1})
	\end{equation*}
	holds for $s\geq1$. It can be check that the well-ordering \eqref{wellorder} satifies conditions (i)-(ii) in Definition \ref{def4_4}.
	
	Assume that there are four vertices $u,$ $u'$, $v$ and $v'\in V(G)$ such that $\{u,u'\}\subset e\in E(G)$, $\{v,v'\}\subset f\in E(G)$, $u\prec v$, $h(u)=h(u')+1$ and $h(v)=h(v')+1$.
	Based on condition (b.2), we get $x_{e\backslash\{u'\}}\geq x_{f\backslash\{v'\}}$. By way of contradiction assume that $v'\prec u'$, then we have $x_{u'}\leq x_{v'}$. If $x_{e\backslash\{u'\}}> x_{f\backslash\{v'\}}$, we let $G'$ be the supertree obtained from $G$ by $e\xrightleftharpoons[v']{u'}f$. From Lemma \ref{2sw}, we have $\rho_{\alpha}(G')>\rho_{\alpha}(G)$, a contradiction. If $x_{e\backslash\{u'\}}= x_{f\backslash\{v'\}}$, by condition (c.2), we get $u'\prec v'$. Given the above, the well-ordering \eqref{wellorder} satisfies condition (iii) in Definition \ref{def4_4}.
	
	From conditions (a.1), (a.2) and the definition of well-ordering \eqref{wellorder}, we see that the well-ordering \eqref{wellorder} satisfies condition (iv) in Definition \ref{def4_4}.
	
	This completes the proof.
	\end{proof}

	By Proposition 3.1 in \cite{XiaoWan and Lu}, the following theorem is easy to be obtained.
	\begin{theorem}\label{deseqmax}
		Let $\pi$ be a non-increasing degree sequence of some $k$-uniform supertrees. $G_{\pi}$ is a unique $k$-uniform supertree with maximum $\alpha$-spectral radius in $\mathbb{T}_{\pi}$.
	\end{theorem}

\section{The $\alpha$-spectral radius of supertrees with given matching number}
	In this section, we mainly determine the unique supertree with maximum $\alpha$-spectral radius among all $k$-uniform supertrees with a given matching number. For a $k$-uniform supertree $G$ with $m$ edges and matching number $\mu$, it follows that $k\mu \leq|V(G)|=n=m(k-1)+1$, implying $\mu\leq\frac{m(k-1)+1}{k}$.
	\begin{definition}\cite{GuoanZhou srh}
		Let $H_{m,k,\mu}$ be a $k$-uniform supertree obtained from a hyperstar $S_{m-\mu+1,k}$ by attaching $\mu-1$ pendent edges to pendent vertices of $\lceil \frac{\mu-1}{k-1}\rceil$ edges in $S_{m-\mu+1,k}$ such that
		\begin{itemize}
			\item[(\romannumeral1)] if $(k-1)\mid(\mu-1)$, then one pendent edge is attached to each pendent vertex of these chosen $ \frac{\mu-1}{k-1}$ edges,
			\item[(\romannumeral2)] and if $(k-1)\nmid(\mu-1)$, then one pendent edge is attached to each pendent vertex of $ \lfloor\frac{\mu-1}{k-1}\rfloor$, and one pendent edge is attached to $(\mu-1)-\lfloor\frac{\mu-1}{k-1}\rfloor(k-1)$ pendent vertices of the remaining chosen one edge.
		\end{itemize} 
	\end{definition}

\begin{figure}[htbp]
	\centering
	\subfigure[Supertree $H_{12,3,8}$]{
		\includegraphics[width=7cm]{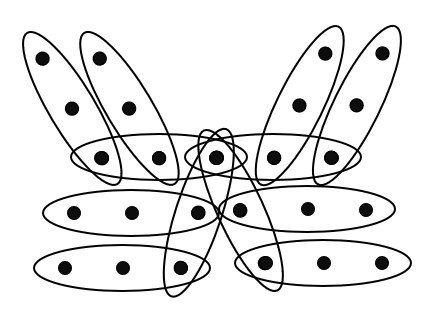}
	}
	\quad
	\subfigure[Supertree $H_{7,4,5}$]{
		\includegraphics[width=5.8cm]{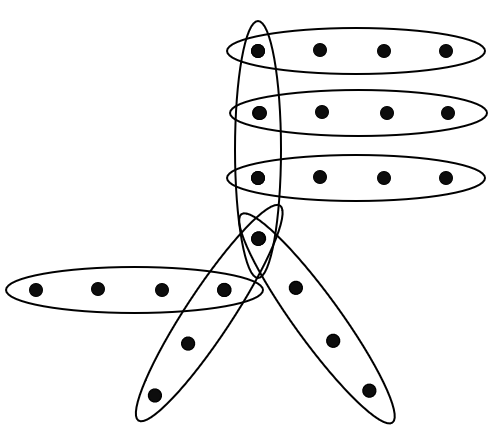}
	}
	\caption{Two supertrees in $H_{m,k,\mu}$.}
	\label{fig2}
\end{figure}

\begin{theorem}
	For $1\leq\mu\leq\lfloor\frac{m(k-1)+1}{k}\rfloor$, $H_{m,k,\mu}$ is the unique supertree with maximum $\alpha$-spectral radius among all $k$-uniform supertrees with $m$ edges and matching number $\mu$.
\end{theorem}
\begin{proof}
	Assume that $G$ be the supertree with maximum $\alpha$-spectral radius among $k$-uniform supertrees with $m$ edges and matching number $\mu$. Let $x$ be the $\alpha$-Perron vector. 
	
	Suppose that there exists a non-pendent edge $e\in E(G)$ in some maximum matching $M$. Take a vertex $u\in e$ and let $G'$ be the supertree obtained from $G$ by an edge-releasing operation on $e$ to $u$. One can check that $G'$ is a $k$-uniform supertree with $m$ edges and $\mu(G')=\mu(G)$. From Lemma \ref{edgere}, we have $\rho_{\alpha}(G')>\rho_{\alpha}(G)$, a contradiction. Thus, all edges in any maximum matching of $G$ are pendent edges.
	
	For a maximum matching $M$ of $G$, we assume that there exists a vertex $u\in V(G)$ with $d_{G}(u)\geq2$ such that for any $e\in M$, $u\notin e$. Let $e_1,\dots,e_t\in E(G)$ be all edges incident with $u$, where $t=d_{G}(u)$. Obviously, $e_i\notin M$ for $1\leq i\leq t$. Suppose that there is an edge $e_i$ such that all vertices of $e_i\backslash\{u\}$ are not in any edge of $M$, then we see that $M\cup \{e_i\}$ is a matching of $G$, a contradiction.
	Without loss of generality, we may assume that there exists a vertex $v_1\in e_1$ such that it's in some edge $e'$ of $M$. Notice that $e'$ is a pendent edge. If $x_{u}\geq x_{v_1}$, let $G''$ be the supertree obtained from $G$ by moving all edges incident with $v_1$ (except $e_1$) from $v_1$ to $u$, and if $x_{u}<x_{v_1}$, let $G'''$ be the supertree obtained from $G$ by moving one edge incident with $u$ (except $e_1$) from $u$ to $v$. It can be check that $G''$ and $G'''$ are both $k$-uniform supertrees with $m$ edges and $\mu(G'')=\mu(G''')=\mu(G)$. By Lemma \ref{moved}, we have $\rho_{\alpha}(G'')>\rho_{\alpha}(G)$ and $\rho_{\alpha}(G''')>\rho_{\alpha}(G)$, a contradiction. Therefore, each vertex with degree at least 2 in $G$ is contained in some edge of $M$.
	
	Suppose that there exist two vertex $u,v\in V(G)$ such that $d_{G}(u)\geq3$ and $d_{G}(v)\geq3$. We may assume that $x_{u}\geq x_{v}$. Then we can obtained a supertree $G''''$ obtained from $G$ by moving an edge not in a maximum matching $M$ incident with $v$ (different from the one containing both $u$ and $v$ if there is such an edge) from $v$ to $u$. It's easy to see that $G''''$ is a $k$-uniform supertrees with $m$ edges and $\mu(G'''')=\mu(G)$. From Lemma \ref{moved}, we have $\rho_{\alpha}(G'''')>\rho_{\alpha}(G)$, a contradiction. Thus, there is at most one vertex with degree at least 3 in $G$.
	
	Based on the above discussion, we obtain the following conclusions. Let $u$ be the vertex with maximum degree in $G$, then 
	\begin{itemize}
		\item for any vertex $v\in V(G)\backslash\{u\}$, $d_{G}(v)=1$ or $d_{G}(v)=2$;
		\item any vertex $v$ with $d_{G}(v)\geq2$ is contained in some edge of $M$;
		\item and all edges in $M$ are pendent edges.  
	\end{itemize}
	
	We see that each edge not in $M$ (if any exists) incident with $u$, implying $d_{G}(u)=m-\mu+1$. Then, one can check that there are $\mu-1$ vertices with degree 2.  Hence, the degree sequence of $G$ is $\pi_0=\{m-\mu+1,\underbrace{2,\dots,2}_{\mu-1},\underbrace{1,\dots,1}_{m(k-1)+1-\mu}\}$. By Theorem \ref{deseqmax}, we have $G\cong G_{\pi_0}\cong H_{m,k,\mu}$.
	
	This completes the proof.
	
\end{proof}
\section*{Declaration of competing interest}
The authors declare that they have no known competing financial interests or personal relationships that could have appeared to influence the work reported in this paper.

\section*{Acknowledgments}\setlength{\baselineskip}{15pt}
This work was supported by the National Natural Science Foundation of China [61773020]. The authors would like to express their sincere gratitude to the referees for their careful reading and insightful suggestions.

\end{document}